\documentclass{article}
\usepackage[utf8]{inputenc}
\usepackage{amssymb}
\usepackage{graphicx}
\usepackage{amsthm} 
\usepackage{amsmath}
\usepackage{caption}
\usepackage{subcaption}

\def \x{\textbf{x}}
\def \y{\textbf{y}}

\def \T{\mathbb{T}_c}

\def \erre{\mathbb{R}}
\def \PP{\mathcal{P}}

\def \BB{\mathcal{B}}

\def \Pimn{\Pi(\mu,\nu)}

\def \erre{\mathbb{R}}

\def \proj#1{(\mathfrak{p}_{#1})_\#}
\def \PP{\mathcal{P}}
\def \mn{(\mu,\nu)}

\newtheorem{theorem}{Theorem}

\newtheorem{corollary}{Corollary}
\newtheorem{lemma}[theorem]{Lemma}

\newtheorem{remark}{Remark}

\newtheorem{example}{Example}
\newtheorem{definition}{Definition}

\title{On the Structure of Optimal Transportation Plans between Discrete Measures}

\author{G. Auricchio and M. Veneroni}
\begin{document}

\maketitle

\begin{abstract}
  In this paper, we prove a structure theorem for discrete optimal transportation plans.  We show that, given any pair of discrete probability measures and a cost function, there exists an optimal transportation plan that can be expressed as the sum of two deterministic plans. As an application, we estimate the infinity-Wasserstein distance between two discrete probability measures $\mu$ and $\nu$ with the $p$-Wasserstein distance, times a constant depending on $\mu$, $\nu$, and the fixed cost function.
\end{abstract}
\vspace{1cm}

\vspace{0.25 cm}
\noindent AMS: 49Q22, 05C70, 39B62.

\normalsize

\section{Introduction}

The Optimal Transport (OT) problem is a classical minimization problem dating back to the work of Monge \cite{Monge1781} and Kantorovich \cite{kantorovich2006translocation,kantorovich1960mathematical}. In this problem, we are given two probability measures, namely $\mu$ and $\nu$, and we search for the cheapest way to reshape $\mu$ into $\nu$. The effort needed in order to perform this transformation depends on a cost function, which describes the underlying geometry of the product space of the support of the two measures. In the right setting, this effort induces a distance between probability measures.

During the last century, the OT problem has been fruitfully used in many applied fields such as the study of systems of particles by Dobrushin \cite{dobrushin1989dynamical}, the Boltzmann equation by Tanaka in \cite{Tanaka1978,Tanaka1973,Murata1974}, and the field of fluidodynamics by Yann Brenier \cite{Brenier1991}. All these results pointed out that
, by a qualitative description of optimal transport, it was possible to gain insightful information on many open problems. For this reason, the Optimal Transport problem has become a topic of major interest for analysts, probabilists and statisticians \cite{Villani2008,AGS,Santambrogio2015}. In particular, a plethora of results concerning the uniqueness \cite{uniqueness,10.2307/827090,figalli2007existence}, the structure \cite{struct,CUESTAALBERTOS199786,RePEc:eee:jmvana:v:32:y:1990:i:1:p:48-54}, and the regularity \cite{loeper2009regularity,bouchitte07} of the optimal transportation plan in the continuous framework has been proved.

In recent years, it has also become a crucial sub-problem in several applications in Computer Vision \cite{Pele2009,Rubner1998,Rubner2000}, Computational Statistic \cite{Levina2001}, Probability \cite{Bassetti2005,RePEc:eee:stapro:v:76:y:2006:i:12:p:1298-1302}, and Machine Learning \cite{pmlr-v70-arjovsky17a,pmlr-v32-solomon14,Frogner2015,Cuturi2014}. However, in these fields, the measures $\mu$ and $\nu$ are discrete, and therefore the optimal transportation plans lack most of the good properties their continuous counterparts enjoy.

In this paper, we study the structure of optimal transportation plans between discrete probability measures. 
After introducing the notion of \emph{trim plan}, between the measures $\mu$ and $\nu$, we prove that such plans are the sum of two deterministic plans, i.e., plans that are induced by the action of two suitable push-forward maps. The first map acts on a portion $\mu^{(d)}$ of $\mu$, while the other one acts on a portion $\nu^{(d)}$ of $\nu$ (Theorem \ref{th:linf_ext}). Thanks to this formula, we recover an extension of the estimate given in \cite{bouchitte07}. Namely, we estimate the \emph{infinity}-Wasserstein distance between a pair of discrete measures $(\mu,\nu)$ (see Definition \ref{def_winfty} below), by the $c$-Wasserstein distance between $\mu$ and $\nu$ times a quantity that only depends on $\mu$ and $\nu$ (Theorem \ref{thm:extensionBJM}).

\section{Basic Notions on Optimal Transport}

In this section, following \cite{Villani2008}, we recall the main definitions regarding optimal transportation and we examine the continuous counterpart \cite{bouchitte07} to our $W^\infty$ estimate.

Given a polish space $(X,d)$, we denote with $\BB(X)$ the set of Borel sets over $X$, while with $\PP(X)$ we denote the set of Borel measures over $X$. Given a Borel measurable function $T:X\to Y$, we denote with $T_\#:\PP(X)\to\PP(Y)$ the push-forward operator induced by $T$, defined by: $(T_\#\mu)[A]=\mu[T^{-1}(A)]$.
The projection maps are $\mathfrak{p}_X:X\times Y \to X$, $\mathfrak{p}_X(x,y)=x$ and $\mathfrak{p}_Y:X\times Y \to Y$,  $\mathfrak{p}_Y(x,y)=y$.

\begin{definition}
Let $\mu$ and $\nu$ be two measures over two polish spaces $X$ and $Y$. The probability measure $\pi \in \PP(X\times Y)$ is a transportation plan between $\mu$ and $\nu$ if
\[
\proj{X}\pi=\mu \quad \quad \text{and} \quad \quad \proj{Y}\pi=\nu.
\]
We denote with $\Pimn$ the set of all the transportation plans between $\mu$ and $\nu.$
\end{definition}

\noindent Given $A\in\BB(X)$ and $B\in\BB(Y)$, the quantity $\pi(A\times B)$ is the amount of mass that travels from the set $A$ to the set $B$. By assigning a cost function $c$ on $X\times Y$ we specify a way to measure the cost of every transportation plan. 

\begin{definition}
\label{def:transp_funct}
Let $\mu\in\PP(X)$, $\nu\in\PP(Y)$, and let $c:X\times Y\to [0,+\infty)$ be a lower semicontinuous (l.s.c.) symmetric cost function. The transportation functional $\T:\Pimn\to [0,+\infty)$ is defined as
\begin{equation}
    \label{eq:transportfunct}
    \T(\pi):=\int_{X\times Y}c\ {\rm d}\pi.
\end{equation}
\end{definition}

\noindent Given two measures $\mu\in\PP(X)$, $\nu\in\PP(Y)$, and a cost function $c$, the optimal transportation problem consists in finding the infimum of $\T$ over $\Pimn$, i.e.
\begin{equation}
\label{eq:transportcost}
    \inf_{\pi \in\Pimn} \T(\pi).
\end{equation}

\noindent By making further assumptions on $c$, it is possible to prove that the infimum in \eqref{eq:transportcost} is actually a minimum. In particular, when the cost function is non negative, the solution exists. For a complete discussion on the existence of the solution, we refer to \cite[Chapter 4]{Villani2008}.\\

\noindent We can use the optimal transportation problem to define a distance over the space $\PP(X)$. In particular, since $X$ is a polish space, we can lift the distance $d$ from $X$ to $\PP(X)$, by choosing $d$ as a cost function in \eqref{eq:transportfunct}.

\begin{definition}
\label{def:Wasdist}
Let $(X,d)$ be a polish space and $p\in [1,\infty)$. The Wasserstein distance of order $p$ between the probability measures $\mu$ and $\nu$ on $X$ is defined as
\begin{equation}
\label{eq:W_p}
    W_p(\mu,\nu):=\Big(\inf_{\pi \in \Pimn}
        \mathbb{T}_{d^p}(\pi)\Big)^{\frac{1}{p}}
    =\Big(\inf_{\pi \in \Pimn}
        \int_{X\times Y}d^p(x,y) {\rm d}\pi(x,y)\Big)^{\frac{1}{p}}.
\end{equation}
When $p=1$, the $1$-Wasserstein distance is also known as Kantorovich-Rubinstein distance.
\end{definition}

When the cost function is not the space distance $d$, we denote the infimum in \eqref{eq:transportcost} with $W_c\mn$.

\begin{remark}
The infimum in \eqref{eq:W_p} could actually be $+\infty$, it is thus customary to restrict $W_p$ to the space of probability measures with finite $p$-moments. 
\end{remark}

\begin{definition}
\label{def_winfty}
Given a cost function $c$, the  $W^{(\infty)}_c$ distance between two measures $\mu$ and $\nu$ is defined as
\begin{equation*}
    W^{(\infty)}_c(\mu,\nu)=\inf_{\pi \in \Pi(\mu,\nu)}|| c ||_{L^{\infty}_\pi} 
\end{equation*}
where $||\,\cdot\,||_{L^{\infty}_\pi}$ is the $L^{\infty}$ norm with respect to the measure $\pi$. When $c$ is the Euclidean distance, we use the notation: $W^{(\infty)}$.
\end{definition}

\noindent Let $\mu$ and $\nu$ be two probability measures on a Lipschitz regular and bounded subset $\Omega\subset\erre^n$. We define the cost function
\[
c_p(\x,\y):=\Bigg(\sqrt{\sum_{i=1}^n|x_i-y_i|^2}\Bigg)^{p}, \quad \quad p>1.
\]
When $\mu$ is absolutely continuous with respect to the Lebesgue measure, it is well known (Theorem 6.3 and Theorem 6.4, \cite{uniqueness}) 
that the optimal transportation plan $\pi$ between $\mu$ and $\nu$ is unique and it is induced by a transportation map $T_p$, i.e.
\[
\pi=(Id,T_p)_\#\mu.
\]
In \cite{bouchitte07}, Bouchitté et al. established an $L^{\infty}_\mu$-bound on the displacement map $Id-T_p$, which only depends on the shape of $\Omega$, on $p$ and on the density of $\mu$. This estimate allowed the authors to give the following upper bound on the $W^{(\infty)}$ distance between $\mu$ and $\nu$.
 
\begin{theorem}[Theorem 1.2, \cite{bouchitte07}]
\label{thm:bouchitte}
Let $\Omega$ be a bounded connected open subset of $\erre^n$ with Lipschitz boundary and denote by $\PP(\overline{\Omega})$ (resp. $\PP_{ac}(\Omega)$) the set of Borel (resp. absolutely continuous) probability measures on $\overline{\Omega}$. Then, for every $p>1$ and every pair $(\mu,\nu)\in\PP_{ac}(\Omega)\times \PP(\overline{\Omega})$ there holds
\begin{equation}
\label{eq:bouchitteestimation}
    (W^{(\infty)}(\mu,\nu))^{p+n}\leq C_{p,n}(\Omega)||f^{-1}||_{L^{\infty}(\Omega)}W^p_p(\mu,\nu),
\end{equation}
where $f$ is the density of $\mu$ with respect to the Lebesgue measure and $C_{p,n}(\Omega)$ is a positive constant depending only on $p,n$, and $\Omega$.
\end{theorem}
 
 \noindent The proof of this result heavily relies on the regularity of $\mu$, hence, when $\mu$ and $\nu$ are both discrete, this result does not apply. In particular, we are no longer able to find a constant depending only on $\mu$ and the geometry of the support of $\mu$, as the following example shows.
 
 \begin{example}
Let $\mu,\nu_\epsilon\in \PP(\erre)$ be defined as
\[
\mu=\dfrac{1}{2}\delta_{0}+\dfrac{1}{2}\delta_{1},\quad\quad\quad \nu_\epsilon=\dfrac{1-\epsilon}{2}\delta_{0}+\dfrac{1+\epsilon}{2}\delta_{1},
\]
for $\epsilon \in (0,1)$, and let $c_2(x,y)=|x-y|^2$. By a simple computation we have that
\[
W^{(\infty)}_2(\mu,\nu_\epsilon)=1,\quad\quad\quad W^2_2(\mu,\nu_\epsilon)=\dfrac{\epsilon}{2}.
\]
Hence, estimate \eqref{eq:bouchitteestimation} does not hold true, as for every constant $C(p,n,\Omega,\mu)>0$ (possibly depending on $p,n,\Omega,\mu$), there exists $\epsilon >0$ such that
\[
   (W^{(\infty)}(\mu,\nu_\epsilon))^{2+1}=1 > C(p,n,\Omega,\mu) W^2_2(\mu,\nu)=\epsilon C(p,n,\Omega,\mu).
\]
\end{example}

\section{Structure of discrete optimal transportation plans}

In what follows, we prove the existence of an optimal transportation plan between two discrete measures that is induced by the action of two push-forward functions, one going from $X$ to $Y$ and one going from $Y$ to $X$. This allows us to establish a bound on $W^{(\infty)}(\mu,\nu)$, similar to the one proved in \cite{bouchitte07}. We always assume $\# X=\# Y =n \in \mathbb{N}$. In this case, we can identify the sets $X$ and $Y$ with $\{1,\dots, n\}$. Without loss of generality, we therefore assume $X=Y$. In this setting, a measure $\mu \in \PP(X)$ has the form $\sum_{x\in X} \mu_x\delta_x$, we thus use the notation $\mu_x$ to denote the coefficient of $\mu$ in $x$ and, likewise, $c_{x,y}$ (resp. $\pi_{x,y}$) stands for the value of $c:X\times Y \to \erre$ (resp. the coefficient of $\pi\in\PP(X\times Y)$) in the point $(x,y)\in X\times Y$.
 
\begin{definition}
Let $\mu,\nu \in \PP(X)$ be two measures on a discrete set $X$ and let $c:X\times X\to \erre$ be a cost function. A minimal solution $\pi^*$ of the transportation problem 
is said to be \emph{trim} if
\[
\# {\rm spt}(\pi^*)\leq \# {\rm spt} (\pi)
\]
for each optimal solution $\pi$.
\end{definition}

\begin{lemma}
\label{lm:trimmingered}
Let $\pi\in\Pi(\mu,\nu)$ be a trim solution. Then each restriction of $\pi$ is a trim solution for its marginals. In particular, if $\pi^{(1)}$ and $\pi^{(2)}$ are such that
\[
\pi=\pi^{(1)}+\pi^{(2)}
\]
and ${\rm spt}(\pi^{(1)})\cap {\rm spt}(\pi^{(2)})=\emptyset$, then $\pi^{(1)}$ and $\pi^{(2)}$ are trim solutions for their marginals.
\end{lemma}

\begin{proof}
Let $\pi^{*}$ be a restriction of $\pi$. By Theorem 4.6 (Chapter 4,  \cite{Villani2008}), 
we know that $\pi^*$ is optimal between its marginals, hence we only need to prove that its support has minimal cardinality. 

Arguing by contradiction, let us assume that $\pi^*$ is not trim, hence there exists another optimal plan $\eta$ between the marginals of $\pi^*$ such that
\[
\#{\rm spt}(\eta)<\#{\rm spt}(\pi^*).
\]
We can define the measure $\hat{\pi}$ as
\[
\hat{\pi}=\pi-\pi^* +\eta,
\]
since $\pi\geq \pi^*$ and $\eta\geq 0$, we have $\hat{\pi}\geq 0$. Moreover, since $\pi^*$ and $\eta$ have the same marginals, $\hat{\pi}$ has the same marginals of $\pi$, therefore $\hat{\pi}\in\Pi(\mu,\nu)$.
Moreover, since $\pi^*$ and $\eta$ are optimal between their marginals, we have
\[
\sum_{(x,y)\in X\times X}c_{x,y}\pi^*_{x,y}=\sum_{(x,y)\in X\times X}c_{x,y}\eta_{x,y},
\]
thus
\begin{eqnarray*}
\sum_{(x,y)\in X\times X}c_{x,y}\hat{\pi}_{x,y}&=&\sum_{(x,y)\in X\times X}c_{x,y}\pi_{x,y}-\sum_{(x,y)\in X\times X}c_{x,y}\pi^*_{x,y}\\
&\quad&+\sum_{(x,y)\in X\times X}c_{x,y}\eta_{x,y}\\
&=&\sum_{(x,y)\in X\times X}c_{x,y}\pi_{x,y}.
\end{eqnarray*}
In particular, $\pi$ and $\hat{\pi}$ have the same cost, therefore $\hat{\pi}$ is an optimal transportation plan between $\mu$ and $\nu$.

To conclude, we notice that, since $\pi^*$ is a restriction of $\pi$, we have
\[
\#{\rm spt}(\pi)=\#{\rm spt}(\pi-\pi^*)+\#{\rm spt}(\pi^*)>\#{\rm spt}(\pi-\pi^*)+\#{\rm spt}(\eta)\geq \#{\rm spt}(\hat{\pi}),
\]
which concludes the contradiction, since $\pi$ is trim by hypothesis.
\end{proof}

 Theorem 6.3 in \cite{uniqueness} states that, whenever $\mu$ is an absolutely continuous measure supported over a compact set $\Omega \subset \erre^n$ and the cost function $c$ is a strictly convex function of the euclidean distance, the optimal transportation plan is induced by a transportation map, regardless of the regularity of $\nu$. When $\mu$ and $\nu$ are both discrete, this result is generally false. However, in the next Theorem \ref{th:linf_ext}, we show that there exists at least one optimal transportation plan between two measures that can be recreated as the action of two functions, one acting from a subset $\tilde{X}\subset {\rm spt}(\mu)$ to $ {\rm spt}(\nu)$ and one acting from a subset $\tilde{Y}\subset {\rm spt}(\nu)$ to ${\rm spt}(\mu)$.

\begin{theorem}
\label{th:linf_ext}
Let $X$ be a discrete polish space and let $\mu$ and $\nu$ be two positive measures over the set $X$ such that
\[
\mu_{a}>0 \quad\quad\quad\forall a \in X,
\]
\[
\nu_b>0 \quad\quad\quad \forall b \in X,
\]
and
\[
\sum_{a\in X }\mu_a=\sum_{b \in X}\nu_b.
\]
Given a cost function $c:X\times X \to \erre$, let $\pi$ be a trim solution of the transportation problem. We can then find two couples of measures $(\mu^{(d)},\mu^{(c)})$ and $(\nu^{(d)},\nu^{(c)})$ and a couple of functions $h^{(1)}$ and $h^{(2)}$ such that
\begin{align}
    \mu &=\mu^{(d)}+\mu^{(c)}\quad \text{and}\quad  \nu=\nu^{(d)}+\nu^{(c)},\label{difmod1}\\
    \pi &=(Id,h^{(1)})_\#\mu^{(d)}+(h^{(2)},Id)_\#\nu^{(d)}.\label{difmod2}
\end{align}
\end{theorem}

We say that the decomposition ensured by Theorem \ref{th:linf_ext} is a \emph{diffusive model} associated with the given (trim) solution $\pi$. We call $\mu^{(d)}$ and $\nu^{(d)}$ the \emph{diffusive part} of $\mu$ and $\nu$, respectively. Similarly, we denote with $\mu^{(c)}$ and $\nu^{(c)}$ the \emph{concentrating part} of $\mu$ and $\nu$, respectively.
Finally, we call $h^{(1)}$ the \emph{diffusive scheme} of $\mu$ and $h^{(2)}$ the \emph{diffusive scheme} of $\nu$.

\begin{proof}
We proceed by induction on the cardinality of $X$. If $\#X=1$, the thesis follows trivially. 

Let us now assume that the statement holds for each couple of measures whose support has cardinality $(n-1)$ and let $\mu$ and $\nu$ be two measures supported on a set with cardinality $n$, namely $X_n$. Given a trim solution $\pi$, it is well known (Chapter 7, \cite{10.5555/248375}) that 
\[
\# {\rm spt} (\pi)\leq 2n-1.
\]
Since $\mu$ and $\nu$ have $n$ points in their support, we can find $\bar{a}\in X$ such that there exists a unique $\bar{b}\in {\rm spt}(\nu)$ for which
\[
\pi_{\bar{a},\bar{b}}>0,
\]
hence $\mu_{\bar{a}}=\pi_{\bar{a},\bar{b}}\leq \nu_{\bar{b}}$. Similarly, we can find $\underline{b}\in X$ such that there exists a unique $\underline{a}\in {\rm spt}(\mu)$ for which 
\[
\pi_{\underline{a},\underline{b}}>0,
\]
so that $\nu_{\underline{b}}=\pi_{\underline{a},\underline{b}}\leq \mu_{\underline{a}}$.

If $\mu_{\bar{a}}=\pi_{\bar{a},\bar{b}} = \nu_{\bar{b}}$, we can restrict the plan $\pi$ to the set ${\rm spt}(\pi)\backslash\{(\bar{a},\bar{b})\}$. We denote this restriction with $\pi_*$. By definition, the marginals of $\pi_*$ are
\[
\mu_*=\mu-\mu_{\bar{a}}\delta_{\bar{a}}
\]
and
\[
\nu_*=\nu-\nu_{\bar{b}}\delta_{\bar{b}}.
\]
In particular, the supports of $\mu_*$ and $\nu_*$ contain $(n-1)$ points each. By induction we can find $(\mu^{(d)}_*,\mu^{(c)}_*)$, $(\nu^{(d)}_*,\nu^{(c)}_*)$, and $(h^{(1)}_*,h^{(2)}_*)$ such that
\[
\mu_*=\mu^{(d)}_*+\mu^{(c)}_*,
\]
\[
\nu_*=\nu^{(d)}_*+\nu^{(c)}_*,
\]
and
\[
\pi_*=(Id,h^{(1)}_*)_\#\mu^{(d)}_* +(h^{(2)}_*,Id)_\#\nu^{(d)}_*.
\]
We can then define
\[
\mu^{(d)}=\mu^{(d)}_*+\mu_{\bar{a}}\delta_{\bar{a}}, \quad\quad\quad \mu^{(c)}=\mu_*^{(c)},
\]
\[
\nu^{(d)}=\nu^{(d)}_*, \quad\quad\quad \nu^{(c)}=\nu_*^{(c)}+\nu_{\bar{b}}\delta_{\bar{b}},
\]
and
\[
h^{(1)}(a)=
\begin{cases}
h^{(1)}_*(a)\quad\quad\quad \text{if }a\neq \bar{a},\\
\bar{b} \quad\quad\quad\quad\quad\; \text{otherwise,}
\end{cases},\quad\quad\quad\quad\quad h^{(2)}(b)=h^{(2)}_*(b).
\]
It easy to see that
\[
\mu=\mu^{(d)}+\mu^{(c)}, \quad \quad \quad \nu=\nu^{(d)}+\nu^{(c)}
\]
and, since $h^{(1)}_\#\delta_{\bar{a}}=\delta_{\bar{b}}$, we have
\begin{equation}
    \label{eq:plansepdiffusive}
    \pi=(Id,h^{(1)})_\#\mu^{(d)}+(h^{(2)},Id)_\#\nu^{(d)},
\end{equation}
which concludes the proof in the case $\mu_{\bar{a}}=\pi_{\bar{a},\bar{b}}=\nu_{\bar{b}}$. We proceed similarly if $\nu_{\underline{b}}=\pi_{\underline{a},\underline{b}}= \mu_{\underline{a}}$.

To conclude, consider the case in which $\mu_{\bar{a}}=\pi_{\bar{a},\bar{b}} < \nu_{\bar{b}}$ and $\nu_{\underline{b}}=\pi_{\underline{a},\underline{b}}< \mu_{\underline{a}}$.
In this case, we restrict $\pi$ to the set ${\rm spt}(\pi)\backslash \{(\bar{a},\bar{b}),(\underline{a},\underline{b})\}$. Let us denote again with $\pi_*$ the restriction and with $\mu_*$ and $\nu_{*}$ its marginals. Since both $\mu_*$ and $\nu_*$ have $(n-1)$ points in their supports, we can again decompose them as
\[
\mu_*=\mu^{(d)}_*+\mu^{(c)}_*, \quad \quad \quad \nu_*=\nu^{(d)}_*+\nu^{(c)}_*
\]
and find a couple of functions $h^{(1)}_*,h^{(2)}_*$ for which
\[
\pi_*=(Id,h^{(1)}_*)_\#\mu^{(d)}_*+(h^{(2)}_*,Id)_\#\nu^{(d)}_*.
\]
We can then define
\[
\mu^{(d)}=\mu^{(d)}_*+\mu_{\bar{a}}\delta_{\bar{a}}, \quad\quad\quad \mu^{(c)}=\mu_*^{(c)}+\mu_{\underline{a}}\delta_{\underline{a}},
\]
\[
\nu^{(d)}=\nu^{(d)}_*{(c)}+\nu_{\underline{b}}\delta_{\underline{b}}, \quad\quad\quad \nu^{(c)}=\nu_*^{(c)}+\nu_{\bar{b}}\delta_{\bar{b}},
\]
and
\[
h^{(1)}(a)=
\begin{cases}
h^{(1)}_*(a)\quad\quad \text{if }a\neq \bar{a},\\
\bar{b} \quad\quad\quad\quad\; \text{otherwise.}
\end{cases}\quad\quad h^{(2)}(b)=\begin{cases}
h^{(2)}_*(b)\quad\quad \text{if }b\neq \underline{b},\\
\underline{a} \quad\quad\quad\quad\; \text{otherwise,}
\end{cases}
\]
which concludes the thesis.
\end{proof}

\begin{remark}
\label{rmk:impossibledecompositions}
Given two measures as in the hypothesis of Theorem \ref{th:linf_ext}, let $\mu^{(d)}$ and $\nu^{(d)}$ be their diffusive part. Since ${\rm spt}(\mu^{(d)})\subset {\rm spt}(\mu)$ and ${\rm spt}(\nu^{(d)})\subset {\rm spt}(\nu)$, the support of the transportation plan defined by formula \eqref{eq:plansepdiffusive} has, at most, $2n$ points. Thus the trim condition on the optimal transportation plan is necessary, as we are going to show in the next example.
\end{remark}

\begin{example}
\label{ex:nottrimdecomp}
Let us take
\[
\mu=\dfrac{1}{4}\bigg(\delta_{(0,0,0)}+\delta_{(1,1,0)}+\delta_{(1,0,1)}+\delta_{(0,1,1)}\bigg)
\]
and
\[
\nu=\dfrac{1}{4}\bigg(\delta_{(1,1,1)}+\delta_{(0,0,1)}+\delta_{(0,1,0)}+\delta_{(1,0,0)}\bigg),
\]
and, as a cost function, we choose the Euclidean distance in $\erre^3$, i.e.
\[
|\x-\y|:=\sqrt{\sum_{i=1}^3(x_i-y_i)^2}.
\]
It is easy to see that the plan
\begin{eqnarray*}
\pi&:=&\dfrac{1}{12}\delta_{(0,0,0)}\otimes\bigg(\delta_{(1,0,0)}+\delta_{(0,1,0)}+\delta_{(0,0,1)}\bigg)\\
&\quad&+\dfrac{1}{12}\delta_{(1,1,0)}\otimes\bigg(\delta_{(0,1,0)}+\delta_{(1,0,0)}+\delta_{(1,1,1)}\bigg)\\
&\quad&+\dfrac{1}{12}\delta_{(1,0,1)}\otimes\bigg(\delta_{(1,0,0)}+\delta_{(0,0,1)}+\delta_{(1,1,1)}\bigg)\\
&\quad&+\dfrac{1}{12}\delta_{(0,1,1)}\otimes\bigg(\delta_{(0,1,0)}+\delta_{(0,0,1)}+\delta_{(1,1,1)}\bigg)\\
\end{eqnarray*}
is optimal. However, according to Remark \ref{rmk:impossibledecompositions}, it cannot be decomposed as in formula \eqref{eq:plansepdiffusive}, since $$\# {\rm spt}(\pi)=12>2\# {\rm spt}(\mu)=8.$$
\end{example}

\begin{figure}[ht!]
\centering
{\renewcommand{\arraystretch}{1}
\setlength{\tabcolsep}{0.1em}
\begin{tabular}{ccc}

  \includegraphics[width=0.45\linewidth]{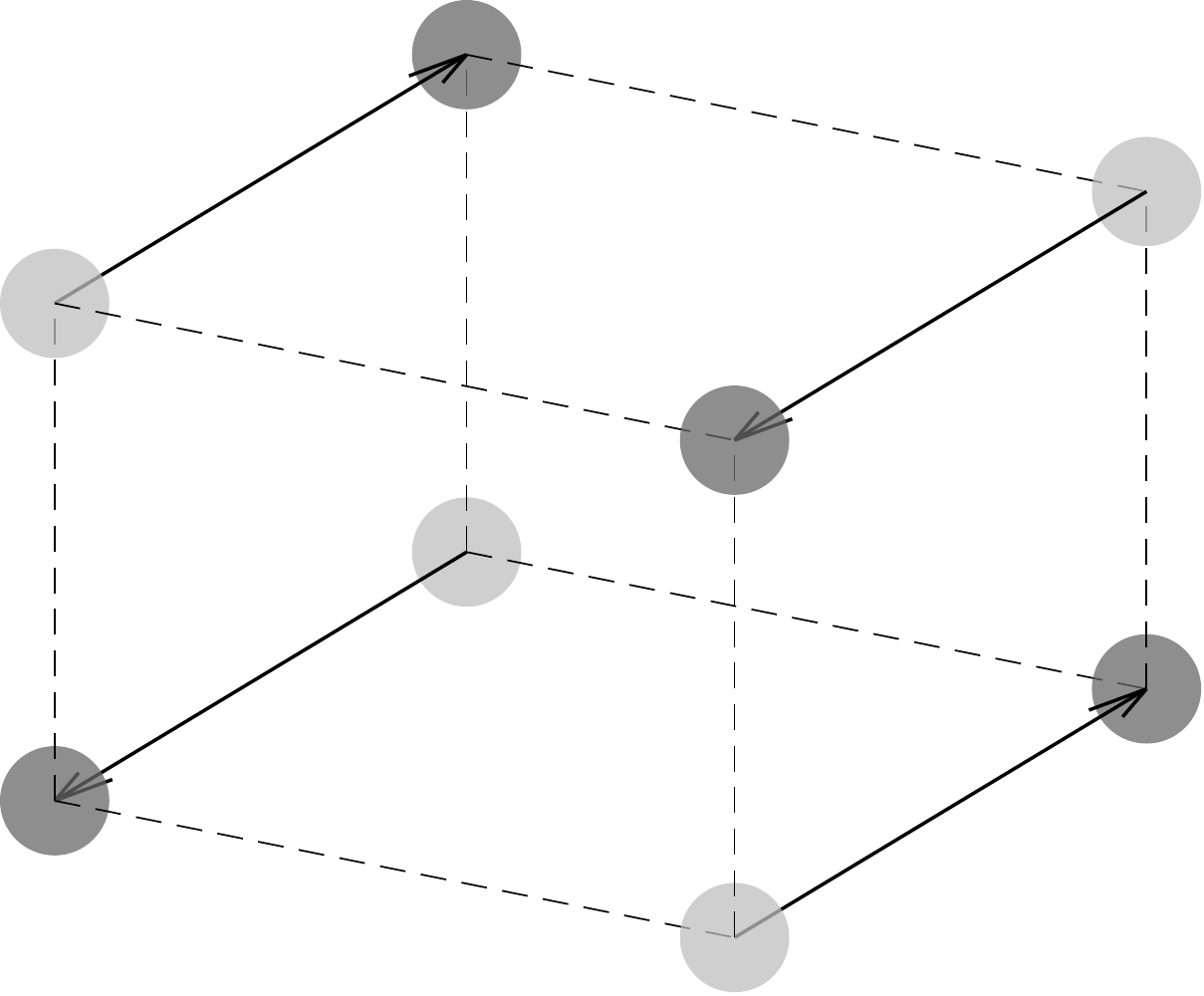} &
  
  \includegraphics[width=0.45\linewidth]{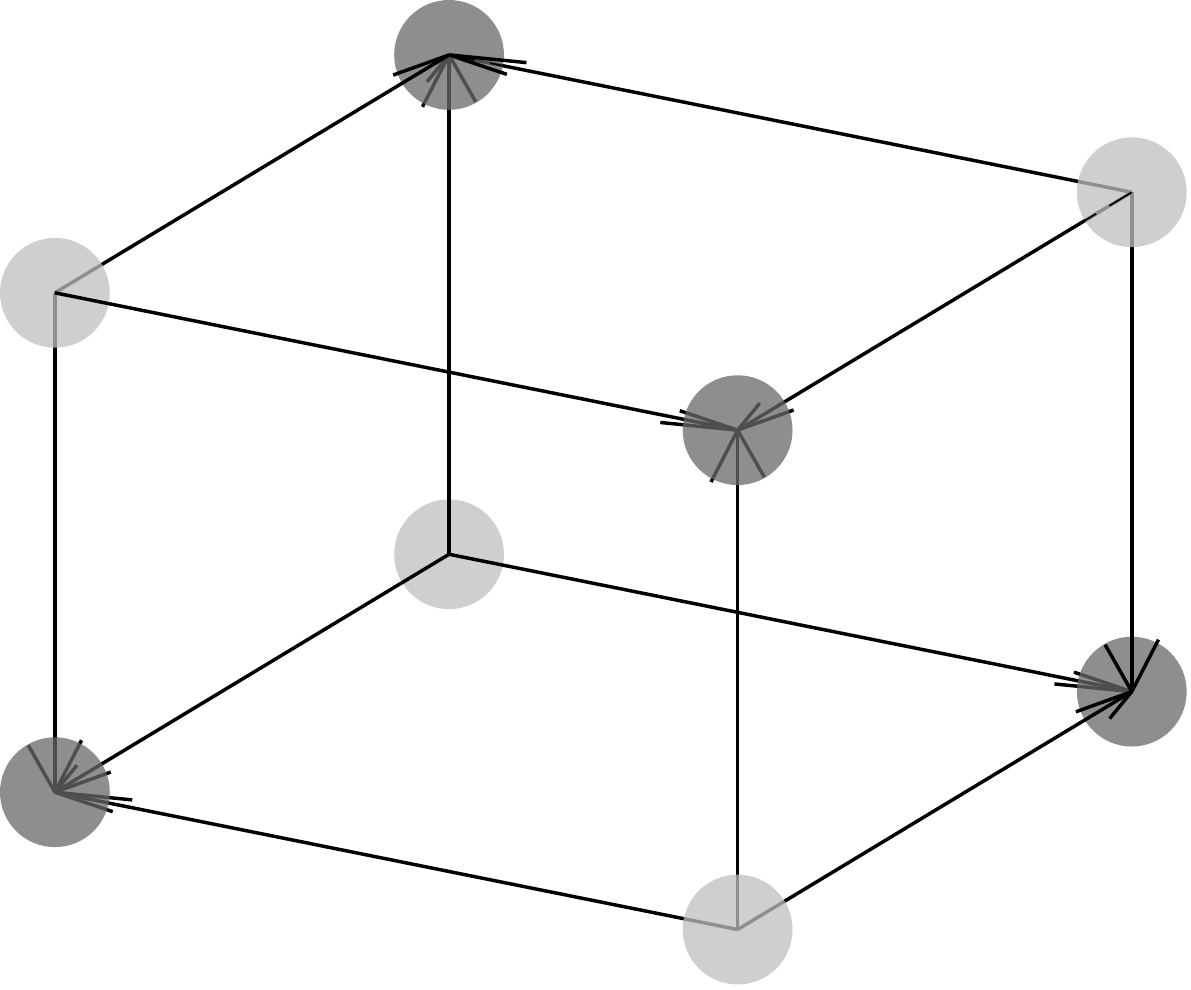} &

 \end{tabular}}
\caption{(Example \ref{ex:nottrimdecomp}) Visual comparison between the optimal plan $\pi$, which is not trim (right) and a trim plan (left). The support of $\mu$ is indicated by light gray dots, the support of $\nu$ by dark gray dots, points $i$ and $j$ are connected by a solid edge if $(i,j)$ belongs to the support of the plan. \label{fig:etrimandnottrim}}
\end{figure}

\begin{remark}
Given a trim solution, there might be more than one diffusive model associated with it. For example, let
\[
\mu=\dfrac{1}{2}\delta_{(0,0)}+\dfrac{1}{2}\delta_{(1,1)}\quad\text{and}\quad \nu=\dfrac{1}{4}\delta_{(-1,1)}+\dfrac{3}{4}\delta_{(1,0)}
\]
be two discrete measures over $\erre^2$. As a cost function, we choose the Euclidean distance
\[
c(\x,\y):=\sqrt{(x_1-y_1)^2+(x_2-y_2)^2}.
\]
Then, the probability measure
\[
\pi=\dfrac{1}{4}\delta_{(0,0)}\otimes\delta_{(-1,1)}+\dfrac{1}{4}\delta_{(0,0)}\otimes\delta_{(1,0)}+\dfrac{1}{2}\delta_{(1,1)}\otimes\delta_{(1,0)}
\]
is a trim plan between $\mu$ and $\nu$. It easy to check that
\begin{eqnarray*}
\mu^{(d)}=\dfrac{1}{4}\delta_{(0,0)}+\dfrac{1}{2}\delta_{(1,1)},\quad &\quad&\quad \mu^{(c)}=\dfrac{1}{4}\delta_{(0,0)},\\
\nu^{(c)}=\dfrac{1}{4}\delta_{(-1,1)}+\dfrac{1}{2}\delta_{(1,0)},\quad &\quad&\quad \nu^{(d)}=\dfrac{1}{4}\delta_{(1,0)},
\end{eqnarray*}
and
\[
h^{(1)}:=\begin{cases}
(-1,1) \quad if \; x=(0,0),\\
(+
1,0) \quad if \; x=(1,1),\\
(0,0) \quad\quad otherwise,
\end{cases}\quad\quad
h^{(2)}(y)=(0,0)\quad  \forall y\in \erre^2,
\]
is a decomposition of the trim plan. However, we can also decompose $\nu$ as
\[
\tilde{\nu}^{(d)}=\dfrac{1}{4}\delta_{(-1,1)},\quad\quad\quad \tilde{\nu}^{(c)}=\dfrac{3}{4}\delta_{(1,0)},
\]
define the functions as
\[
h^{(1)}(\x)=(1,0)\quad \forall \x \in \erre^2, \quad\quad\quad h^{(2)}(\y)=(0,0) \quad \forall \y \in \erre^2,
\]
and still obtain an admissible decomposition of $\pi$.
\end{remark}

\section{An Upper Bound for the Infinity Wasserstein distance in the Discrete Setting}

As an immediate consequence of the diffusive model decomposition \eqref{difmod1}-\eqref{difmod2} given in Theorem \ref{th:linf_ext}, we can decompose the Wasserstein distance associated to a cost function $c$ and use it to estimate the infinity-Wasserstein distance.

\begin{corollary}
\label{cor:winfty}
Let $\mu,\nu\in\PP(X)$ be two discrete measures, $c:X\times X \to \erre$ be a cost function, and $\pi$ be a trim solution of the transportation problem. Given a diffusive model for $\pi$, we have
\[
W_c(\mu,\nu)=\sum_{x\in X}c(x,h^{(1)}(x))\mu^{(d)}_x+\sum_{y\in X}c(h^{(2)}(y),y)\nu^{(d)}_y
\]
and
\[
\mathbb{T}_c^{(\infty)}(\pi)=\max \bigg\{||c(x,h^{(1)}(x))||_{L_{\mu^{(d)}}^{\infty}},||c(h^{(2)}(y),y)||_{L^{\infty}_{\nu^{(d)}}}\bigg\}.
\]
In particular, we have
\begin{equation}
\label{eq:linfextim}
  W_c(\mu,\nu)\geq \alpha W^{(\infty)}_c(\mu,\nu),  
\end{equation}
where 
\begin{equation}
\label{eq:alphamin}
    \alpha=\min_{a\in {\rm spt}(\mu^{(d)}),b\in {\rm spt}(\nu^{(d)})}\{\nu^{(d)}_b,\mu^{(d)}_a\}.
\end{equation}

\end{corollary}


The value $\alpha$ defined in \eqref{eq:alphamin} depends on the particular diffusive model we choose. However, since $W_c(\mu,\nu)$ and $W^{(\infty)}_c$ do not depend on the choice of the diffusive model, if we can give a lower bound on $\alpha$ for a particular diffusive model, we can generalize the estimate \eqref{eq:linfextim}.

\begin{corollary}
\label{crl:alphaunif}
Let $\mu,\nu\in \PP(X)$ be two discrete measures and $c:X    \times X\to \erre_+$ be a cost function. For any trim plan $\pi$, there exists a diffusive model for which
\begin{equation}
    \label{eq:boundonmudi}
\alpha \geq \min_{(A,B)\in K(\mu,\nu)}\bigg\{\bigg|\sum_{x\in A}\mu_x-\sum_{y\in B}\nu_y\bigg|\bigg\},
\end{equation}
where $\alpha$ is defined in relation \eqref{eq:alphamin} and
\[
K(\mu,\nu):=\bigg\{(A,B)\subset X\times X\quad \text{s.t.}\quad \bigg|\sum_{x\in A}\mu_x-\sum_{y\in B}\nu_y\bigg|>0\bigg\}.
\]
\end{corollary}

\begin{proof}
Let $n$ be the cardinality of $X$. Since $\pi$ is trim between $\mu$ and $\nu$, we have $\#{\rm spt}(\pi)\leq 2n-1$, hence we can find $\bar{x}_1$ such that 
\[
\exists ! \; \;\bar{y}_1 \quad s.t. \quad \pi_{\bar{x}_1,\bar{y}_1}\neq 0
\]
and $\underline{y}_1$ such that
\[
\exists !\;  \;  \underline{x}_1 \quad s.t. \quad \pi_{\underline{x}_1,\underline{y}_1}\neq 0.
\]
If $\underline{x}_1=\bar{x}_1$ (and hence $\underline{y}_1=\bar{y}_1$), we have $\mu_{\bar{x}_1}=\nu_{\bar{y}_1}$ and we define
\[
\mu^{(d)}_{\bar{x}_1}=\mu_{\bar{x}_1}, \quad \quad \quad \nu^{(c)}_{\bar{y}_1}=\mu_{\bar{x}_1},
\]
and
\[
\mu^{(1)}:=\mu-\mu_{\bar{x}_1}\delta_{\bar{x}_1},\quad  \nu^{(1)}:=\nu-\nu_{\bar{y}_1}\delta_{\bar{y}_1}, \quad \pi^{(1)}=\pi-\pi_{\bar{x}_1,\bar{y}_1}\delta_{\bar{x}_1,\bar{y}_1}.
\]
Otherwise, if $\underline{x}_1\neq\bar{x}_1$ (and hence $\underline{y}_1\neq\bar{y}_1$), we set
\begin{eqnarray*}
\mu^{(d)}_{\bar{x}_1}=\mu_{\bar{x}_1},\quad &\quad& \quad \mu^{(c)}_{\underline{x}_1}=\nu_{\underline{y}_1},\\
\nu^{(d)}_{\underline{y}_1}=\nu_{\underline{y}_1},\quad &\quad& \quad \nu^{(c)}_{\bar{y}_1}=\mu_{\bar{x}_1},
\end{eqnarray*}
and
\begin{eqnarray*}
\mu^{(1)}&=&\mu-\mu_{\bar{x}_1}\delta_{\bar{x}_1}-\nu_{\underline{y}_1}\delta_{\underline{x}_1},\\
\nu^{(1)}&=&\nu-\nu_{\underline{y}_1}\delta_{\underline{y}_1}-\mu_{\bar{x}_1}\delta_{\bar{y}_1}\\
\pi^{(1)}&=&\pi-\pi_{\bar{x}_1,\bar{y}_1}\delta_{\bar{x}_1,\bar{y}_1}-\pi_{\underline{x}_1,\underline{y}_1}\delta_{\underline{x}_1,\underline{y}_1}.
\end{eqnarray*}
    In both cases, we find two measures, $\mu^{(1)}$ and $\nu^{(1)}$, whose support has, at most, $n-1$ points. Since $\pi^{(1)}$ is a restriction of a trim plan, by Lemma \ref{lm:trimmingered}, also $\pi^{(1)}$ is trim between its marginals $\mu^{(1)}$ and $\nu^{(1)}$. Therefore, we can repeat the process, finding two points $\bar{x}_2$ and $\underline{y}_2$ for which
\[
\exists ! \;\; \bar{y}_2 \quad s.t. \quad \pi_{\bar{x}_2,\bar{y}_2}\neq 0
\]
and
\[
\exists ! \;\; \underline{x}_2\quad s.t. \quad \pi_{\underline{x}_2,\underline{y}_2}\neq 0.
\]
We can then extend the definition of the measures $\mu^{(d)},\mu^{(c)},\nu^{(d)}$, and $\nu^{(c)}$, define the measures $\mu^{(2)}$, $\nu^{(2)}$,  and $\pi^{(2)}$ and start all over again.

At each step, we define two measures $\mu^{(i)}$ and $\nu^{(i)}$ and increase the cardinality of the supports of $\mu^{(d)},\mu^{(c)},\nu^{(d)}$, and $\nu^{(c)}$. Given any $x \in {\rm spt}(\mu^{(d)})$, we can then find $i\in \{0,1,\dots,n-1\}$ such that
\begin{equation}
    \label{eq:mu_mu_k}
    \mu^{(d)}_x=\mu^{(i)}_x,
\end{equation}
and, similarly, for any $y\in {\rm spt}(\nu^{(d)})$, we can find a $j\in \{0,1,\dots,n-1\}$ such that
\[
\nu^{(d)}_y=\nu^{(j)}_y,
\]
with the convention $\mu^{(0)}=\mu$ and $\nu^{(0)}=\nu$.
The relation between $\mu^{(i)}$ and $\mu^{(i+1)}$ is either
\[
\mu^{(i+1)}=\mu^{(i)}-\mu^{(i)}_{\bar{x}_{i+1}}\delta_{\bar{x}_{i+1}}
\]
or
\[
\mu^{(i+1)}=\mu^{(i)}-\mu^{(i)}_{\bar{x}_{i+1}}\delta_{\bar{x}_{i+1}}-\nu^{(i)}_{\underline{y}_{i+1}}\delta_{\underline{x}_{i+1}}.
\]
Similarly, we have
\[
\nu^{(i+1)}=\nu^{(i)}-\nu^{(i)}_{\underline{y}_{i+1}}\delta_{\underline{y}_{i+1}}
\]
or
\[
\nu^{(i+1)}=\nu^{(i)}-\nu^{(i)}_{\underline{y}_{i+1}}\delta_{\underline{y}_{i+1}}-\mu^{(i)}_{\bar{x}_{i+1}}\delta_{\bar{y}_{i+1}}.
\]
Similarly, we can write $\mu^{(i)}$ and $\nu^{(i)}$ as a function of $\mu^{(i-1)}$ and $\nu^{(i-1)}$, and then express $\mu^{(i+1)}$ through $\mu^{(i-1)}$ and $\nu^{(i-1)}$ as
\begin{equation}
\label{eq.muipiuuno}
   \mu^{(i+1)}_x=\sum_{a\in\tilde{A}_2}\mu^{(i-1)}_a-\sum_{b\in\tilde{B}_2}\nu^{(i-1)}_b, 
\end{equation}
where $\tilde{A}_2$ and $\tilde{B}_2$ are two subsets of $X$ whose cardinality is at most two. By iterating this process, we are able to find
\begin{equation}
    \label{eq:muipiuuno2}
    \mu^{(i+1)}_x=\sum_{a\in\tilde{A}_{n-(i+1)}}\mu_a-\sum_{b\in\tilde{B}_{n-(i+1)}}\nu_b,
\end{equation}
where $\tilde{A}_{n-(i+1)}$ and $\tilde{B}_{n-(i+1)}$ are subsets of $X$,  whose cardinality is $n-(i+1)$. Since the left side of \eqref{eq.muipiuuno} is positive, we can rewrite \eqref{eq:muipiuuno2} as
\begin{equation}
\label{eq:mu_definitiva}
    \mu^{(i+1)}_x=\bigg|\sum_{a\in\tilde{A}_2}\mu^{(i-1)}_a-\sum_{b\in\tilde{B}_2}\nu^{(i-1)}_b\bigg|.
\end{equation}
By taking the minimum over $K(\mu,\nu)$ of the right side in \eqref{eq:mu_definitiva}, we find
\[
\mu^{(i)}_x\geq \min_{(A,B)\in K(\mu,\nu)}\bigg\{\bigg|\sum_{x\in A}\mu_x-\sum_{y\in B}\nu_y\bigg|\bigg\},
\]
for any $i \in \{0,1,\dots,n-1\}$ and each $x\in {\rm spt}(\mu^{(i)})$, therefore, from relation \eqref{eq:mu_mu_k}, we get
\[
\mu^{(d)}\geq \min_{(A,B)\in K(\mu,\nu)}\bigg\{\bigg|\sum_{x\in A}\mu_x-\sum_{y\in B}\nu_y\bigg|\bigg\}.
\]
Similarly, one can prove
\[
\nu_y^{(d)}\geq \min_{(A,B)\in K(\mu,\nu)}\bigg\{\bigg|\sum_{x\in A}\mu_x-\sum_{y\in B}\nu_y\bigg|\bigg\},
\]
for each $y \in {\rm spt}(\nu^{(d)})$, hence relation \eqref{eq:boundonmudi} is proven.

\end{proof}

\begin{figure}[!t]
\centering
{\renewcommand{\arraystretch}{1}
\setlength{\tabcolsep}{0.1em}
\begin{tabular}{ccc}

    \includegraphics[height=0.24\linewidth]{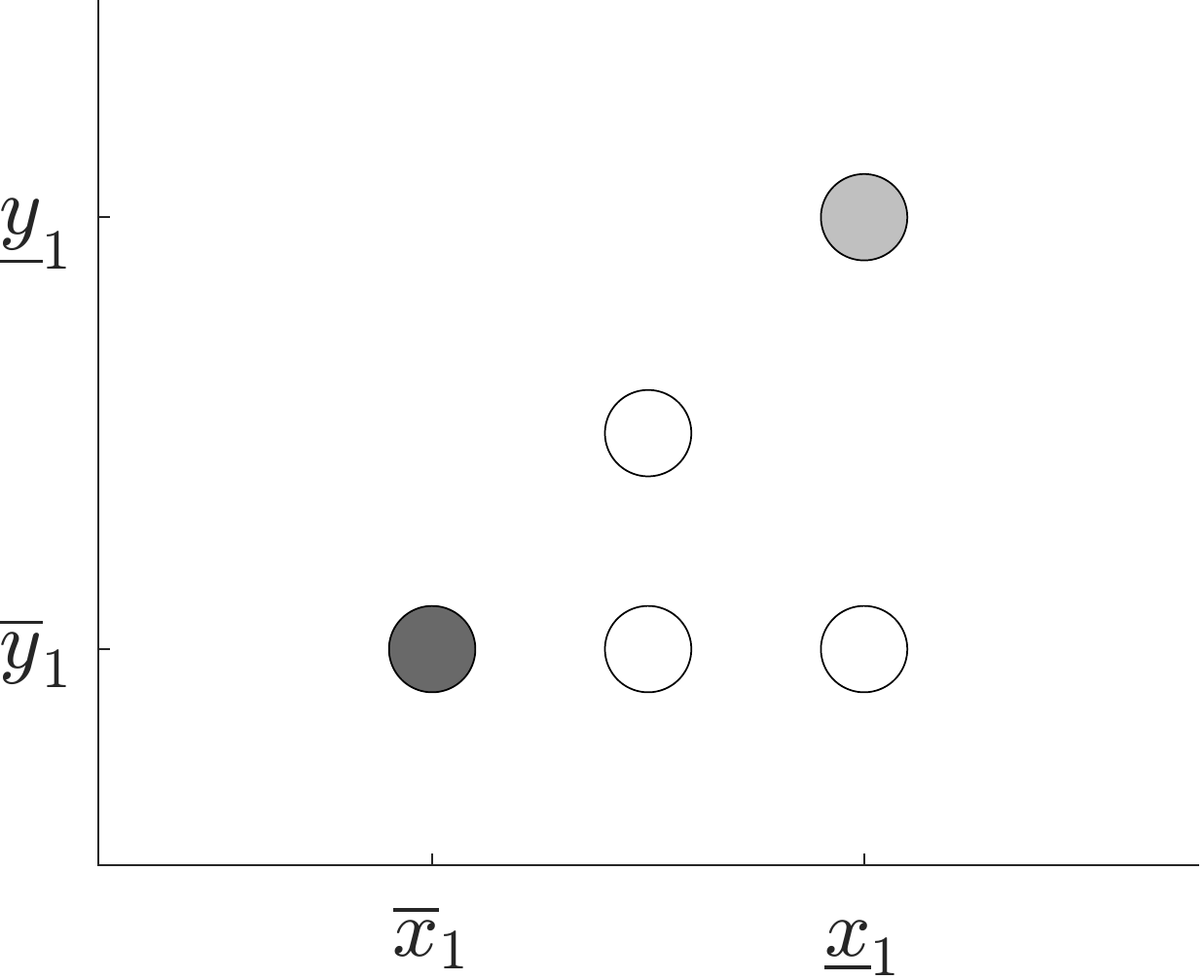} &
  
  \includegraphics[height=0.24\linewidth]{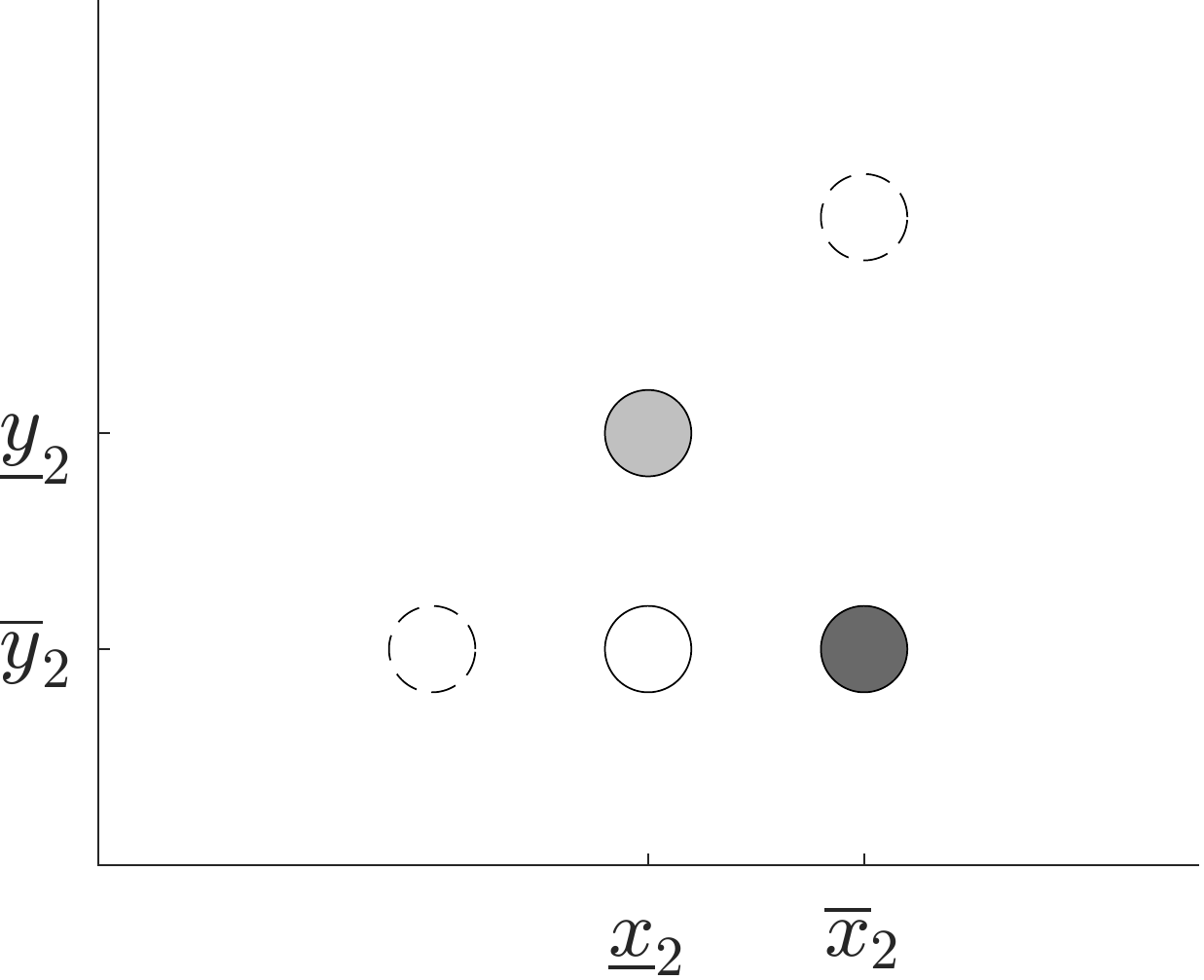} &
 
  \includegraphics[height=0.24\linewidth]{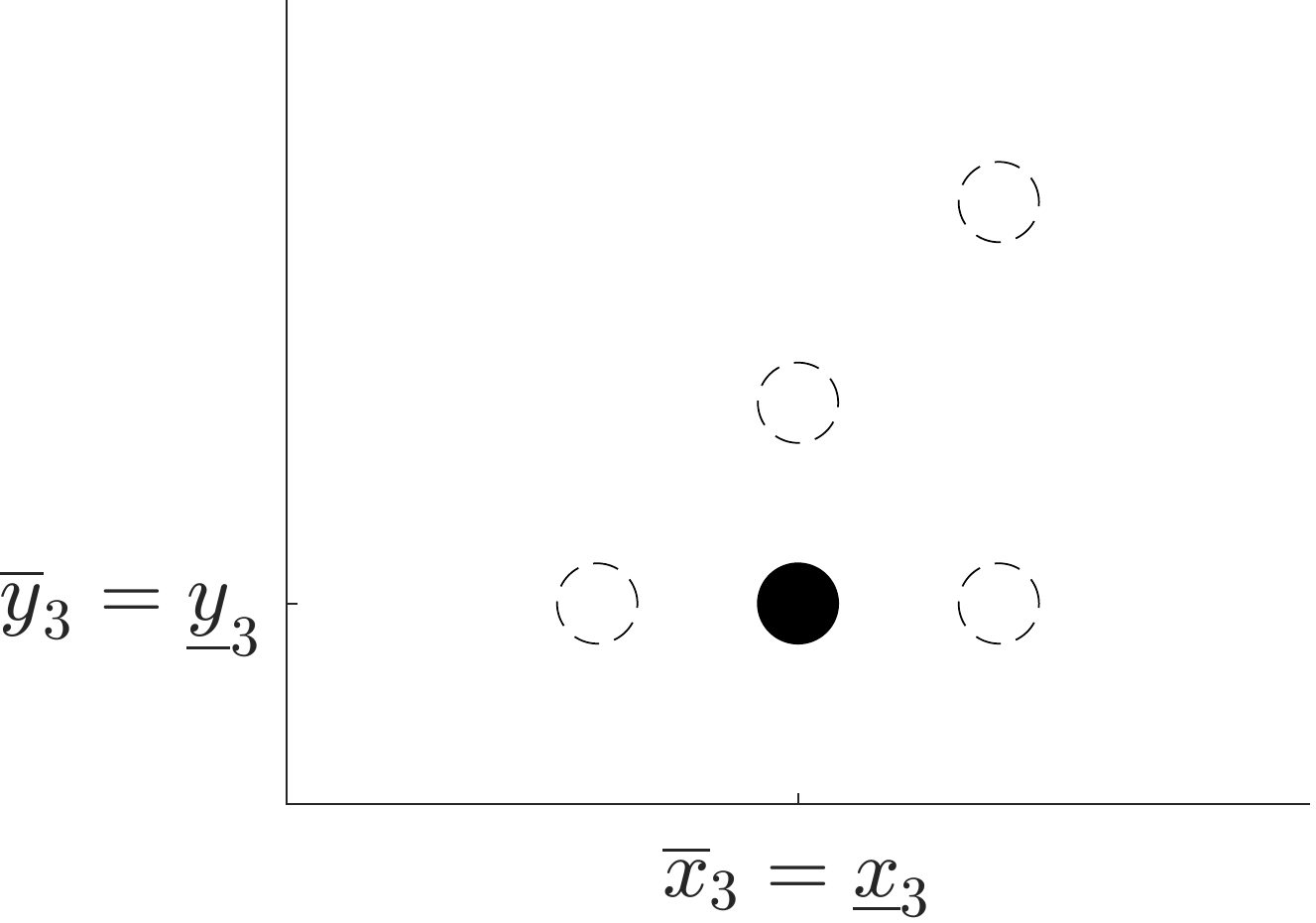}

 \end{tabular}}
\caption{Visual description of the decomposition process used in proof of Theorem \ref{th:linf_ext}. Here, the measures $\mu$ and $\nu$ are one-dimensional and have $3$ points each in their supports. \label{fig:explaindec}}
\end{figure}


In Corollary \ref{cor:winfty}, we bound $W_c^{(\infty)}$ from above with $W_c$.  However, due to the properties of $W^{(\infty)}_c$, it is possible to relate this distance to the Wasserstein cost induced by any $p-$power of the same cost function.

\begin{lemma}
\label{thm:pWinf}
Let $\mu,\nu\in\PP(X)$ and let $c:X\times X \to \erre_+$ be a cost function. Given any $p>0$, it holds true
\[
W^{(\infty)}_{c^p}(\mu,\nu)=\big(W^{(\infty)}_c(\mu,\nu)\big)^p.
\]
\end{lemma}

\begin{proof}
Let $\pi\in\Pi(\mu,\nu)$ be a plan such that
\[
T_{c}(\pi)=W_c^{(\infty)}(\mu,\nu),
\]
then
\[
W^{(\infty)}_{c^p}(\mu,\nu)\leq T_{c^p}(\pi)=T_c(\pi)^p=\big(W_c^{(\infty)}(\mu,\nu)\big)^p.
\]
Similarly, one can prove $\big(W_c^{(\infty)}(\mu,\nu)\big)^p\leq W^{(\infty)}_{c^p}(\mu,\nu)$ and conclude the thesis.
\end{proof}
Thanks to Lemma \ref{thm:pWinf}, we are able to prove the following result.

\begin{theorem}
\label{thm:extensionBJM}
Given a cost function $c:X\times X\to [0,\infty)$, let $\mu,\nu\in \PP(X)$ be two discrete measures. For any $p\geq 1$, 
\begin{equation}
\label{eq:extensionBJM}
    W^{(\infty)}_{c}(\mu,\nu)\leq\frac{W_{c_p}(\mu,\nu)}{(\alpha_p)^{\frac{1}{p}}},
\end{equation}
where $\alpha_p$ is the constant defined in \eqref{eq:alphamin}.
\end{theorem}

\begin{proof}
Given a $p\geq 1$, let us denote with $\pi^{(p)}$ the trim optimal transportation plan between $\mu$ and $\nu$ according to the cost function $c_p$. Given a diffusive model for $\pi^{(p)}$, we denote with $\alpha_p$ the constant defined in \eqref{eq:alphamin}. From Lemma \ref{thm:pWinf} we have
\[
W^{(\infty)}_{c_p}(\mu,\nu)=(W^{(\infty)}_c(\mu,\nu))^p,
\]
hence, for any $p$, we have
\[
(W^{(\infty)}_c(\mu,\nu))^p=W^{(\infty)}_{c_p}(\mu,\nu)\leq\frac{W^p_{c_p}(\mu,\nu)}{\alpha_p},
\]
i.e.,
\[
W^{(\infty)}_c(\mu,\nu)\leq\frac{W_{c_p}(\mu,\nu)}{(\alpha_p)^{\frac{1}{p}}}.
\]
\end{proof}

In particular, since the constant $\alpha$ from Corollary \ref{crl:alphaunif} bounds from below every $\alpha_p$ and does not depend on the cost function but only on the starting measures $\mu$ and $\nu$, we have
\[
W^{(\infty)}_c(\mu,\nu)\leq\frac{W_{c_p}(\mu,\nu)}{(\alpha)^{\frac{1}{p}}}
\]
for any $p\geq 1$. In particular, if we take
\[
c(\x,\y):=\sqrt{\sum_{i=1}^n|x_i-y_i|^2},
\]
we recover the bound proposed in Theorem \ref{thm:bouchitte} for discrete measures.

\begin{remark}
The estimate in \eqref{eq:extensionBJM} is sharp. To prove it, let us take
\[
\mu=\delta_a \quad\quad \text{and}\quad\quad \nu=\delta_b
\]
where $a,b\in \erre^n$. By definition \eqref{eq:alphamin}, we have $\alpha =1$. Moreover, it is easy to see that
\[
W^{(\infty)}(\mu,\nu)=|a-b|\quad \quad \text{and}\quad \quad W_p(\mu,\nu)=|a-b|,
\]
which proves the sharpness of inequality \eqref{eq:linfextim}.
\end{remark}

\subsubsection*{Acknowledgements}
We are deeply indebted to Filippo Santambrogio for introducing us to the work of Bouchitté, Jimenez, and Mahadevan and for several stimulating discussions and valuable suggestions. We thank Stefano Gualandi for his feedback and Gabriele Loli for enhancing the images of this paper.

\bibliographystyle{plain}
\bibliography{ref.bib}

\begin{thebibliography}{10}

\bibitem{struct}
Taoufiq Abdellaoui and Henri Heinich.
\newblock Caract\'erisation d'une solution optimale au probl\`eme de
  {M}onge-{K}antorovitch.
\newblock {\em Bulletin de la Soci\'et\'e Math\'ematique de France},
  127(3):429--443, 1999.

\bibitem{CUESTAALBERTOS199786}
J.A.~Cuesta Albertos, C.~Matr\'an, and A.~Tuero-Dı\a'az.
\newblock On the monotonicity of optimal transportation plans.
\newblock {\em Journal of Mathematical Analysis and Applications},
  215(1):86--94, 1997.

\bibitem{AGS}
Luigi Ambrosio, Nicola Gigli, and Giuseppe Savar{\'e}.
\newblock {\em Gradient flows: In metric spaces and in the space of probability
  measures}.
\newblock Birkh\"auser Basel, 2008.

\bibitem{pmlr-v70-arjovsky17a}
Martin Arjovsky, Soumith Chintala, and L{\'e}on Bottou.
\newblock {W}asserstein generative adversarial networks.
\newblock {\em Proceedings of Machine Learning Research}, 70:214--223, 06--11
  Aug 2017.

\bibitem{RePEc:eee:stapro:v:76:y:2006:i:12:p:1298-1302}
Federico Bassetti, Antonella Bodini, and Eugenio Regazzini.
\newblock On minimum {K}antorovich distance estimators.
\newblock {\em Statistics and Probability Letters}, 76(12):1298--1302, 2006.

\bibitem{Bassetti2005}
Federico Bassetti and Eugenio Regazzini.
\newblock Asymptotic properties and robustness of minimum dissimilarity
  estimators of location-scale parameters.
\newblock {\em Society for Industrial and Applied Mathematics}, 50:312--330, 01
  2005.

\bibitem{bouchitte07}
Guy Bouchitté, Chloé Jimenez, and Rajesh Mahadevan.
\newblock A new ${L}^{\infty}$ estimate in optimal mass transport.
\newblock {\em Proceedings of the American Mathematical Society},
  135:3525--3535, 11 2007.

\bibitem{Brenier1991}
Yann Brenier.
\newblock On the translocation of masses.
\newblock {\em Communications on pure and applied mathematics}, 44(4):375--417,
  1991.

\bibitem{10.2307/827090}
Luis~A. Caffarelli, Mikhail Feldman, and Robert~J. McCann.
\newblock Constructing optimal maps for {M}onge's transport problem as a limit
  of strictly convex costs.
\newblock {\em Journal of the American Mathematical Society}, 15(1):1--26,
  2002.

\bibitem{Cuturi2014}
Marco Cuturi and Arnaud Doucet.
\newblock Fast computation of {W}asserstein barycenters.
\newblock {\em Proceedings of Machine Learning Research}, 32(2):685--693,
  22--24 Jun 2014.

\bibitem{10.5555/248375}
George~B. Dantzig and Mukund~N. Thapa.
\newblock {\em Linear Programming 1: Introduction}.
\newblock Springer-Verlag, Berlin, Heidelberg, 1997.

\bibitem{dobrushin1989dynamical}
Roland Dobrushin.
\newblock Vlasov equations.
\newblock {\em Funct. Anal. Appl.}, 13(2):115–123, 1979.

\bibitem{figalli2007existence}
Alessio Figalli.
\newblock Existence, uniqueness, and regularity of optimal transport maps.
\newblock {\em SIAM journal on mathematical analysis}, 39(1):126--137, 2007.

\bibitem{Frogner2015}
Charlie Frogner, Chiyuan Zhang, Hossein Mobahi, Mauricio Araya, and Tomaso~A
  Poggio.
\newblock Learning with a {W}asserstein loss.
\newblock In {\em Advances in Neural Information Processing Systems}, pages
  2053--2061, 2015.

\bibitem{uniqueness}
Wilfred Gangbo and Robert~J. McCann.
\newblock The geometry of optimal transportation.
\newblock {\em Acta Mathematica}, 177(177):113–161, 1996.

\bibitem{kantorovich1960mathematical}
Leonid~V. Kantorovich.
\newblock Mathematical methods of organizing and planning production.
\newblock {\em Management science}, 6(4):366--422, 1960.

\bibitem{kantorovich2006translocation}
Leonid~V. Kantorovich.
\newblock On the translocation of masses.
\newblock {\em Journal of Mathematical Sciences}, 133(4):1381--1382, 2006.

\bibitem{Levina2001}
Elizaveta Levina and Peter Bickel.
\newblock The {E}arth {M}over's {D}istance is the {M}allows distance: Some
  insights from statistics.
\newblock {\em Proceedings of the IEEE International Conference on Computer
  Vision}, 2:251 -- 256 vol.2, 02 2001.

\bibitem{loeper2009regularity}
Gr{\'e}goire Loeper et~al.
\newblock On the regularity of solutions of optimal transportation problems.
\newblock {\em Acta mathematica}, 202(2):241--283, 2009.

\bibitem{Monge1781}
Gaspard Monge.
\newblock M{\'e}moire sur la th{\'e}orie des d{\'e}blais et des remblais.
\newblock {\em Histoire de l'Acad{\'e}mie Royale des Sciences de Paris}, 1781.

\bibitem{Murata1974}
Hiroshi Murata, Hiroshi;~Tanaka.
\newblock An inequality for certain functional of multidimensional probability
  distributions.
\newblock {\em Hiroshima Math}, 4(1):75--81, 1974.

\bibitem{Pele2009}
Ofir Pele and Michael Werman.
\newblock Fast and robust {E}arth {M}over's {D}istances.
\newblock In {\em 2009 IEEE 12th International Conference on Computer Vision},
  pages 460--467. IEEE, 2009.

\bibitem{Rubner1998}
Yossi Rubner, Carlo Tomasi, and Leonidas Guibas.
\newblock Metric for distributions with applications to image databases.
\newblock {\em Proceedings of the IEEE International Conference on Computer
  Vision}, pages 59--66, 02 1998.

\bibitem{Rubner2000}
Yossi Rubner, Carlo Tomasi, and Leonidas~J Guibas.
\newblock The {E}arth {M}over's {D}istance as a metric for image retrieval.
\newblock {\em International Journal of Computer Vision}, 40(2):99--121, 2000.

\bibitem{RePEc:eee:jmvana:v:32:y:1990:i:1:p:48-54}
L.~Rüschendorf and S.~T. Rachev.
\newblock {A characterization of random variables with minimum L2-distance}.
\newblock {\em Journal of Multivariate Analysis}, 32(1):48--54, 1990.

\bibitem{Santambrogio2015}
Filippo Santambrogio.
\newblock Optimal transport for applied mathematicians.
\newblock {\em Birk{\"a}user, NY}, pages 99--102, 2015.

\bibitem{pmlr-v32-solomon14}
Justin Solomon, Raif Rustamov, Leonidas Guibas, and Adrian Butscher.
\newblock Wasserstein propagation for semi-supervised learning.
\newblock {\em Proceedings of Machine Learning Research}, 32(1):306--314,
  22--24 Jun 2014.

\bibitem{Tanaka1973}
Hiroshi Tanaka.
\newblock An inequality for a functional of probability distributions and its
  application to {K}ac's one-dimensional model of a {M}axwellian gas.
\newblock {\em Zeitschrift für Wahrscheinlichkeitstheorie und Verwandte
  Gebiete}, 27:47–52, 1973.

\bibitem{Tanaka1978}
Hiroshi Tanaka.
\newblock Probabilistic treatment of the {B}oltzmann equation of {M}axwellian
  molecules.
\newblock {\em Zeitschrift für Wahrscheinlichkeitstheorie und Verwandte
  Gebiete}, 46:67–105, 1978.

\bibitem{Villani2008}
C{\'e}dric Villani.
\newblock {\em Optimal transport: old and new}, volume 338.
\newblock Springer-Verlag, Berlin Heidelberg, 2008.

\end{thebibliography}
\end{document}